\documentclass[a4paper,11pt]{article}
\usepackage[left = 2cm, right = 2cm, top = 2cm, bottom = 2cm]{geometry}

\usepackage{amsmath,amsthm,amssymb, amsfonts, mathrsfs}
\usepackage{mathtools}

\usepackage[hidelinks]{hyperref}
\usepackage[nameinlink,capitalise,noabbrev]{cleveref}

\usepackage{natbib}
\theoremstyle{definition}
\newtheorem{definition}{Definition}[section]

\newtheorem{remark}[definition]{Remark}

\theoremstyle{plain}
\newtheorem{conjecture}[definition]{Conjecture}
\newtheorem{theorem}[definition]{Theorem}
\newtheorem{lemma}[definition]{Lemma}

\newtheorem{corollary}[definition]{Corollary}

\def \cl {\colon}
\def \ce {\coloneqq}

\def \E {\mathbb{E}}
\def \F {\mathbb{F}}

\def \Z {\mathbb{Z}}

\renewcommand{\le}{\leqslant}
\renewcommand{\ge}{\geqslant}

\def \eps {\varepsilon}

\def \mF{\mathcal{F}}

\newcommand{\eqdef}{\coloneq}

\newcommand {\R} {\mathcal{R}}
\newcommand {\rplus} { +_{\R} }
\newcommand {\rtplus} { +_{\widetilde{\R}} }
\newcommand {\Fp} {\F_p}
\newcommand {\D} {\Delta}

\title{\vspace{-3ex}On restricted sumsets with bounded degree relations}

\author{
	Minghui Ouyang\thanks{School of Mathematical Sciences, Peking University, Beijing 100871, China. \texttt{ouyangminghui1998@gmail.com}. }
}

\date{}

\begin{document}
	\maketitle
	\vspace{-5ex}
	\begin{abstract}
		Given two subsets $A, B \subseteq \mathbb{F}_p$ and a binary relation $\mathcal{R} \subseteq A \times B$, the restricted sumset of $A, B$ with respect to $\mathcal{R}$ is defined as $A +_{\mathcal{R}} B = \{ a+b \colon (a,b) \notin \mathcal{R} \}$. When $\mathcal{R}$ is taken as the equality relation, determining the minimum value of $|A +_{\mathcal{R}} B|$ is the famous Erd\H{o}s--Heilbronn problem, which was solved separately by Dias da Silva, Hamidoune and Alon, Nathanson and Ruzsa. Lev later conjectured that if $A, B \subseteq \mathbb{F}_p$ with $|A| + |B| \le p$ and $\mathcal{R}$ is a matching between subsets of $A$ and $B$, then $|A +_{\mathcal{R}} B| \ge |A| + |B| - 3$. 

		We confirm this conjecture in the case where $|A| + |B| \le (1-\varepsilon)p$ for any $\varepsilon > 0$, provided that $p > p_0$ for some sufficiently large $p_0$ depending only on $\varepsilon$. Our proof builds on a recent work by Bollob\'as, Leader, and Tiba, and a rectifiability argument developed by Green and Ruzsa. Furthermore, our method extends to cases when $\mathcal{R}$ is a degree-bounded relation, either on both sides $A$ and $B$ or solely on the smaller set. 
		
		In addition, we construct subsets $A \subseteq \mathbb{F}_p$ with $|A| = \frac{6p}{11} - O(1)$ such that $|A +_{\mathcal{R}} A| = p-3$ for any prime number $p$, where $\mathcal{R}$ is a matching on $A$. This extends an earlier construction by Lev and highlights a distinction between the combinatorial notion of the restricted sumset and the classcial Erd\H{o}s--Heilbronn problem, where $|A +_{\mathcal{R}} A| \ge p$ holds given $\mathcal{R} = \{(a,a) \colon a \in A\}$ is the equality relation on $A$ and $|A| \ge \frac{p+3}{2}$. 

	\end{abstract}
	
\section{Introduction}
	Given a prime number $p$ and two subsets $A, B \subseteq \Fp$, the sumset $A+B \eqdef \{a + b \cl a \in A, b \in B\}$ is defined as the collection of all pairwise sums of elements from the Cartesian product $A \times B$. The classical Cauchy--Davenport theorem states that for any two non-empty subsets $A, B$ of $\Fp$, we have $|A+B| \ge \min\{p, |A|+|B|-1\}$. A natural generalization of the Cauchy--Davenport theorem is to exclude a few pairs from $A \times B$ and consider the sumset formed by the remaining pairs. 
	
	In 1964, Erd\H{o}s and Heilbronn~\cite{EH64} conjectured that if $A$ is a non-empty subset of $\Fp$, then $|A \dot{+} A| \ge \min\{p, 2|A|-3\}$, where $A \dot{+} A \eqdef \{a+b \cl a,b \in A,\, a \neq b\}$. Dias da Silva and Hamidoune~\cite{DH94} confirmed this conjecture using the exterior algebra method in 1994. Later, Alon, Nathanson, and Ruzsa~\cite{ANR96} extended the Erd\H{o}s--Heilbronn conjecture to a non-symmetric setting using the polynomial method, particularly through the Combinatorial Nullstellensatz. Their method applies not only to the restricted sumset of distinct pairs but also to restricted sumsets with any restriction specified by a low-degree polynomial $P(x,y)$, i.e. $A \rplus B = \{a + b \cl P(a,b) \neq 0\}$. For further details on this approach, see~\cite{ANR96, ANR95, Alo99, Kar09, PS02}. 
	
	In this paper, we investigate the size of the restricted sumset with respect to any degree-bounded, combinatorially defined relation between $A$ and $B$. Lev~\cite{Lev00-ii} defined the \emph{restricted sumset} as
	
	\begin{definition}[Restricted Sumset] \label{def:restricted_sumset}
		Suppose $A, B$ are two subsets of $\Fp$ or $\Z$, and $\R \subseteq A \times B$ is a binary relation between $A$ and $B$. 
		The \emph{restricted sumset} of $A$ and $B$ is defined as
			\[ A \rplus B \ce \left\{ a+b \cl a \in A, b \in B, (a,b) \notin \R \right\}. \]
	\end{definition}
	
	One should note that $\R$ consists of the set of \emph{forbidden pairs}, rather than the set of pairs allowed under the addition operation, which was adopted in some other context. For more reference about research on this kind of restricted sumsets, see \cite{Lev00-i, Lev00-ii, Lev01-iii}, where Lev~\cite{Lev00-ii} investigates the minimum size of $A \rplus B$ according to $|A|, |B|$ and $|\R|$ for an arbitrary relation $\R$, and~\cite{Lev01-iii} under a notion of \emph{$(K,s)$-regular} on $\R$, which requires the maximum degree of $\R$ on both sides $A$ and $B$ to be at most $s$, and every element $c$ has at most $K$ representations as $c = a + b$ for $(a,b) \in \R$. 
	
	As a special case, when $\R$ is taken as an injective function from $B$ to $A$, i.e. $\R$ is a matching between subsets of $A$ and $B$, Lev~\cite{Lev00-ii} proposed the following conjecture. 
	
	\begin{conjecture} [Lev~\cite{Lev00-ii}] \label{conj:Lev}
		Suppose $p$ is a prime number, $A, B \subseteq \Fp$, $|B| \le |A|$, and $\R \cl B \to A$ is an injective function from $B$ to $A$. Then 
			\[ |A \rplus B| \ge 
				\left\{ \begin{array}{cl}
				|A| + |B| - 3& \quad \text{if } |A| + |B| \le p \\
				p-3& \quad \text{if } |A| + |B| = p+1 \\
				p-2& \quad \text{if } |A| + |B| \ge p+2. 
			\end{array} \right.  \]
	\end{conjecture}
	
	The second case in the statement is necessary because of an example constructed by Lev~\cite{Lev00-ii}, which shows that there exists $A \subseteq \Fp$ and a matching $\R$ on $A$ satisfying $|A| = \frac{p+1}{2}$ and $|A \rplus A| = p-3$. We later give an example (\Cref{thm:counterexample_Lev}) with $A \subseteq \Fp$ that extends this construction, where $|A| = \frac{6p}{11} - O(1)$ and $|A \rplus A| \le p-3$. Hence, we should be more cautious concerning the restricted sumset in $|A| + |B| \ge p$ regime. 
	
	For symmetric restricted sumset, Bollob\'as, Leader, and Tiba~\cite{BLT22} established a deep result on the size of the restricted sumset when replacing one summand of $A+A$ by a subset $A' \subseteq A$ of constant size. Our proof is based on a few auxiliary results in the proof of their result. 
	
	\begin{theorem} [{\cite[Theorem 35]{BLT22}}] \label{thm:BLT_EH_extension}
		Let $\{d(n)\}_{n=1}^{\infty}$ be a sequence of natural numbers with $d(n) = o(n)$, and let $\eps>0$. Then there are integers $c$ and $n_0$ such that the following holds.
		
		Let $A\subseteq \Fp$ with $n_0 \le |A| = n \le (1 - \eps)p/2$, and let $\R \subseteq A \times A$ have degree at most $d(n)$. Then there is a set $A' \subseteq A$ of size at most $c$, such that $|A \rplus A'|\ge 2|A|-1-2d$. 
	\end{theorem}
	
	As a corollary, this confirms Lev's conjecture under the assumption $A = B$ in $|A| \le (1-\eps)p/2$ regime. 
	
	\begin{corollary} [{\cite[Theorem 9]{BLT22}}] \label{thm:BLT_EH_simple}
		For each $\eps > 0$ and integer $d$ there is an $n_0$ such that the following holds. Let $A$ be a subset of $\Fp$ with $n_0 \le |A| < (1-\eps)p/2$, and $\R \subseteq A \times A$ is symmetric and has degree at most $d$. Then $|A \rplus A| \ge 2|A| - 1 - 2d$. 
	\end{corollary}
	
	In \Cref{thm:main_Fp_case_strong}, we generalize this theorem to a non-symmetric setting. 
	
\subsection{Results}
	We first prove the following result about restricted sumsets with respect to a bounded degree relation for subsets of $\Z$. The proof is quite straightforward and not the main focus of this paper. However, since this result closely resembles its counterpart in $\Fp$, which is the primary focus of our study, we include it here for completeness. 
	
	\begin{theorem} [$\Z$ case] \label{thm:main_Z_case}
		Suppose $\D \ge 1$ is an integer, $A, B \subseteq \Z$ satisfy $|B| \le |A|$, and $\R \subseteq A \times B$ is a binary relation between $A$ and $B$. 
		
		(i) If the maximum degree of $\R$ on $B$ is at most $\D$, then
			\[ |A \rplus B| \ge |A| + |B| - 3\D. \]
		
		(ii) If the maximum degree of $\R$ on both $A$ and $B$ is at most $\D$, then
			\[ |A \rplus B| \ge |A| + |B| - 1 - 2\D. \]
	\end{theorem}
	
	For the $\Fp$ case, we have the following result.
	
	\begin{theorem} [$\Fp$ case] \label{thm:main_Fp_case}
		For any $\eps > 0$ and integer $\D \ge 1$, there exist constants $c_\eps > 0$ and an integer $p_0$, such that the following holds for any prime number $p$. Suppose $A, B \subseteq \Fp$ satisfy $|B| \le |A|$, $|A| + |B| \le (1-\eps)p$, $|B| \le c_\eps p$, and $\R \subseteq A \times B$ is a binary relation between $A$ and $B$. Then, provided that $p > p_0$ or $|A| \ge 10\D$, we have
		
		(i) If the maximum degree of $\R$ on $B$ is at most $\D$, then
			\[ |A \rplus B| \ge |A| + |B| - 3\D. \]
		
		(ii) If the maximum degree of $\R$ on both $A$ and $B$ is at most $\D$, then
			\[ |A \rplus B| \ge |A| + |B| - 1 - 2\D. \]
		
		Moreover, $c_\eps$ and $p_0$ can be chosen as $c_\eps = \left(1 + \frac{1}{\alpha}\right)^{-1} \left( 8(2\D + 3 + \frac{1}{\alpha}) \right)^{-24(2\D + 3 + \frac{1}{\alpha})^4}$, $p_0 = 4^{10\D}$, where $\alpha = 2^{-18} \cdot 3.1 \cdot 10^{-1549} \eps$. 
	\end{theorem}
	
	The condition $p > p_0$ for some constant $p_0$ is more natural and easier to state compared to $|A| \ge 10\D$. However, the latter can be particularly useful when $\R$ is dense, i.e. $\D = \Theta(p)$, as otherwise the requirement $p > p_0$ may force $p$ to be exponentially large in terms of $\D$. We therefore impose both conditions here. If either condition holds, the conclusion follows. 
	
	When the degree is bounded on both sides, we have the following stronger result. We do not combine \Cref{thm:main_Fp_case_strong} with \Cref{thm:main_Fp_case}(ii) because the dependence of $\beta, p_0$ on $\eps, \D$ in \Cref{thm:main_Fp_case_strong} is more intricate than that the corresponding parameters $c_\eps$ and $p_0$ in \Cref{thm:main_Fp_case}. Determining their explict forms requires a more detailed examination of the references. 
	
	\begin{theorem} [$\Fp$ case] \label{thm:main_Fp_case_strong}
		For any $\eps > 0$ there exists $\beta > 0$ such that for any integer $\D \ge 1$, there exists an integer $p_0$ with the following property: If $p$ is a prime number, $A, B \subseteq \Fp$ satisfy $|A| + |B| \le (1-\eps)p$, and $\R \subseteq A \times B$ is a binary relation between $A$ and $B$ with maximum degree (on both sides) at most $\D$, then provided that $|A| \ge \beta\D$ or $p > p_0$, we have
			\[ |A \rplus B| \ge |A| + |B| - 1 - 2\D. \]
		
		Moreover, $p_0$ can be chosen as $p_0 = \exp(O_\eps(\D))$ which we shall determine in the proof.
	\end{theorem}
	
	It is noteworthy that Bollob\'as, Leader, and Tiba~\cite[Theorem 5]{BLT22} proved a related theorem, showing that whenever $|A| + |B| \le (1-\eps)p$ and $|B|$ is sufficiently small relative to $|A|$, then the optimal lower bound for the sumset, as given by the Cauchy--Davenport theorem, can be attained using just three elements from $B$, i.e. $|A + \{b_1, b_2, b_3\}| \ge |A| + |B| - 1$ for some $b_1, b_2, b_3 \in B$. 
	
	A direct application of this theorem could give us a weaker bound $|A \rplus B| \ge |A| + |B| - 3\D - 1$ in \Cref{thm:main_Fp_case}(i). Thus, our main objective is to improve the $O(\D)$ term in this inequality, which turns out to be rather delicate. Apart from harnessing some of the ideas from their proof, we also utilise a rectifiability argument given by Green and Ruzsa~\cite{GR06} to deal with the case when the size of $A$ and $B$ are comparable. 
	
	As a corollary, we confirm \Cref{conj:Lev} proposed by Lev whenever $|A| + |B| \le (1-\eps)p$ for sufficiently large $p$. 
	
	\begin{corollary} \label{coro:Lev_conj}
		(i) For any $\eps > 0$, let $c_\eps = \left( \frac{\eps}{10^{1600}} \right)^{\frac{10^{6400}}{\eps^4}}$. 
		Suppose $p$ is a prime number, $A, B \subseteq \Fp$ satisfy $|A| + |B| \le (1-\eps)p$, $|B| \le |A|$, $|B| \le c_\eps p$, and $f \cl B \to A$ is an arbitrary function. We have
			\[ \Big| \big\{ a+b \cl a \in A, b \in B, a \neq f(b) \big\} \Big| \ge |A| + |B| - 3. \]
		
		(ii) For any $\eps > 0$, there exists $p_0$ such that for any prime number $p > p_0$, $A, B \subseteq \Fp$ satisfy $|A| + |B| \le (1-\eps)p$, and $\R$ is a matching between some elements of $A$ and $B$. We have
			\[ \Big| \big\{ a+b \cl a \in A, b \in B, (a,b) \notin \R \big\} \Big| \ge |A| + |B| - 3. \]
	\end{corollary}
	
	A natural question is how tight \Cref{thm:main_Fp_case} is. \Cref{thm:main_Fp_case}(ii) is already tight even in $\Z$, and the bound in (i) is weaker than (ii) whenever $\D \ge 2$. One would hope that we could prove $|A \rplus B| \ge |A| + |B| -1 - 2\D$ for any relation $\R$ with bounded degree $\D$ from the smaller side $B$ under some reasonable assumption. However, the following example shows that this is too optimistic, even in the case where $|A| = |B|$ and $A, B \subseteq \Z$. 
	
	\begin{theorem} \label{thm:Z_construction}
		Given integers $\D \ge 1$ and $n \ge 2\D$. There exist $A, B \subseteq \Z$, $|A| = |B| = n$ and a relation $\R \subseteq A \times B$ with bounded degree $\D$ from $B$ such that
			\[ |A \rplus B| = |A| + |B| - 1 - \left\lfloor \frac{5\D}{2} \right\rfloor. \]
	\end{theorem}
	
	We believe the above theorem is tight, which suggests \Cref{thm:main_Z_case}(i) could potentially be strengthened to $|A \rplus B| \ge |A| + |B| - 1 - \left\lfloor \frac{5\D}{2} \right\rfloor$. It is worth noting that any improvement of the $-3\D$ term in \Cref{thm:main_Z_case}(i) would directly lead to a strengthening of \Cref{thm:main_Fp_case}(i), with the same parameters $c_\eps, p_0$ unchanged. The bottleneck of the current method for potential improvement of \Cref{thm:main_Fp_case}(i) lies in the case $A, B \subseteq \Z$ after applying the rectifiability argument. The remaining part of the proof would remain valid without any modification. 
	
	\begin{conjecture}
		Suppose $\D \ge 1$ is an integer, $A, B \subseteq \Z$ satisfies $|B| \le |A|$, and $\R \subseteq A \times B$ is a binary relation between $A$ and $B$. 
		If the maximum degree of $\R$ on $B$ is at most $\D$, we have
			\[ |A \rplus B| \ge |A| + |B| - 1 - \left\lfloor \frac{5\D}{2} \right\rfloor. \]
	\end{conjecture}
	
	The second part of this paper presents several constructions, two of which serve the main purpose. \Cref{coro:counterexample_Lev}(ii) explains why the additional requirement $|B| = O_\eps(p)$ is necessary for \Cref{coro:Lev_conj}(i) to hold; \Cref{thm:counterexample_Lev}(iii) demonstrates why a stronger assumption $|A| + |B| \ge (1+\delta)p$ is required for \Cref{conj:Lev} to hold in order to prove $|A \rplus B| \ge p-2$, even in the case where $|B| \le \eps p$. 
	
	\Cref{thm:counterexample_Lev}(i) restates a construction originally given by Lev~\cite{Lev00-ii}. We reformulated it here for consistency with the notation and style used throughout this paper. \Cref{thm:counterexample_Lev}(ii) was inspired from the same paper. 
	
	\begin{theorem} \label{thm:counterexample_Lev}
		Suppose $p$ is a prime number. 
		
		(i) For any integer $k \ge 1$ and $1 \le \ell \le \lfloor \frac{p-k-1}{2k-1} \rfloor$, there exist subsets $A, B \subseteq \Fp$ and a function $R \cl B \to A$ such that $|A| = p - (k-1) \ell - k + 1$, $|B| = k\ell + 2$ and $|A \rplus B| = p-k$. 
		
		(ii) For any prime number $p$, there exists a subset $A \subseteq \Fp$ and a symmetric relation $\R \subseteq A \times A$ with maximum degree $1$, such that $|A| = 6 \lfloor \frac{p}{11} \rfloor - 3$ and $|A \rplus A| = p-3$. 
		
		(iii) For any $\eps > 0$, there exists $\delta > 0$ such that for any sufficiently large prime number $p$, there exist $A, B \subseteq \Fp$ with $|A| + |B| > (1+\delta)p - O(1)$, $|B| \le \eps p$, and a relation $\R \subseteq A \times B$ with maximum degree $1$, such that $|A \rplus B| = p-3$. 
	\end{theorem}
	
	Plugging $\ell = \lfloor \frac{p-k-1}{2k-1} \rfloor$ and $\ell = 1$ into \Cref{thm:counterexample_Lev}(i), we have the following. 
	
	\begin{corollary} \label{coro:counterexample_Lev}
		Suppose $p$ is a prime number. 
		
		(i) For any integer $k \ge 1$, there exist subsets $A, B \subseteq \Fp$ and a function $R \cl B \to A$ such that $|A| = p - (k-1) \lfloor \frac{p-k-1}{2k-1} \rfloor - k + 1$, $|B| = k \lfloor \frac{p-k-1}{2k-1} \rfloor + 2$ and $|A \rplus B| = p-k$. 
		
		(ii) For any $\eps > 0$, there exist $A, B \subseteq \Fp$ and a function $R \cl B \to A$ such that $|A| = (1-2\eps)p + O(1)$, $|B| = \eps p + O(1)$ and $|A \rplus B| = |A| + |B| - 4$, where the value of the $O(1)$ term within the expressions of $|A|$ and $|B|$ are smaller than $3$. 
	\end{corollary}
	
	Thus, for potential extensions of \Cref{coro:Lev_conj}(i), in the regime where $|A| + |B| \ge p$, it is necessary to assume at least $|A| + |B| \ge p-k+\lfloor \frac{p-k-1}{2k-1} \rfloor + 4 > \frac{2k}{2k-1} p - k + 2$ in order to establish $|A \rplus B| \ge p-k+1$. In the regime $|A| + |B| \le (1-\eps)p$, an additional assumption $|B| \le \eps p$ is required. 
	
	Combining \Cref{coro:Lev_conj}(i) and the constructions above, we propose the following conjecture. The main obstacle in proving this conjecture lies in the fact that, in \Cref{thm:main_Fp_case}, the parameter $c_\eps$ is not linear in $\eps$. So it remains unclear how to prove $|A \rplus B| \ge |A| + |B| - 3$ under assumptions such as $|A| + c|B| < p$ for any constant $c$. 
	
	\begin{conjecture} \label{conj:our_conj_after_Lev}
		Suppose $p$ is a prime number, $A, B \subseteq \Fp$ with $|B| \le |A|$. Let $\R \cl B \to A$ be an arbitrary function from $B$ to $A$. If $|A| + 2|B| \le p$, then $|A \rplus B| \ge |A| + |B| - 3$. 
	\end{conjecture}
	
	In the case where both sides have bounded degree $1$, \Cref{thm:counterexample_Lev}(ii) and (iii) indicate that extra care is required when extending \Cref{conj:Lev} beyond the $|A| + |B| \le p$ regime. In particular, a stronger assumption $|A| + |B| \ge (1+\delta)p$ is necessary even in the case $|B| \le \eps p$ in order to prove $|A \rplus B| \ge p-2$. 
	
	At the end of this paper, we examine the tightness of the construction in \Cref{coro:counterexample_Lev}(i). Specifically, we investigate the minimum possible value of $|A| + |B|$ required to ensure that $|A \rplus B| \ge p - k + 1$. We have the following result, with a proof that follows from a straightforward counting argument. 
	
	\begin{theorem} \label{thm:sum_greater_than_p_case}
		Suppose $p$ is a prime number and $k \ge 1$ is an integer. For any subsets $A, B \subseteq \Fp$, let $\R \cl B \to A$ be an arbitrary function. If $|A| + |B| \ge \lfloor \frac{2kp}{2k-1} \rfloor + 1$, then $|A \rplus B| \ge p-k+1$. 
	\end{theorem}
	
	This shows that the construction in \Cref{coro:counterexample_Lev}(i) is tight up to an additive term of $k-2$ in $|A| + |B|$. 
	
	\paragraph{Notations.} 
	Throughout this paper, we assume that $p$ is a prime number, $A, B$ are finite subsets of either $\Z$ or $\Fp$, as will be clarified in context. 
	Additionally, we consider $\R \subseteq A \times B$ to be a binary relation between $A$ and $B$. 
	
	We say that $\R$ has \emph{bounded degree $\D$} if, when viewed as a bipartite graph with parts $(A, B)$, its maximum degree is at most $\D$. Similarly, we say that $\R$ has \emph{bounded degree $\D$ on $B$} if the maximum degree among elements in $B$ is at most $\D$. 
	
	\paragraph{Paper organization.}	
	In \Cref{sec:main_proof}, we prove \Cref{thm:main_Z_case}, \Cref{thm:main_Fp_case} and \Cref{thm:main_Fp_case_strong}. Then, in \Cref{sec:constructions}, we establish \Cref{thm:Z_construction}, \Cref{thm:counterexample_Lev}, and \Cref{thm:sum_greater_than_p_case}, along with additional discussions. 
	
\section{Preliminaries}
	We first state the celebrated Pl\"{u}nnecke--Ruzsa theorem here. 
	
	\begin{theorem}[Pl\"{u}nnecke--Ruzsa \cite{Plu70, Ruz89, Pet12}] \label{thm:Plunnecke--Ruzsa}
		Let $A, B$ be finite non-empty subsets of an abelian group, and suppose that $|A+B| \le K |A|$. Then, for all non-negative integers $m, n$ we have $|mB - nB| \le K^{m+n} |A|$, where $mB$ denote the $m$-fold sumset $\underbrace{B + \cdots + B}_{m}$. 
	\end{theorem}
	
	We also need the notion of rectifiability. 
	
	\begin{definition}[Rectifiability] \label{def:rectifiable}
		Suppose $p$ is a prime number and $k \ge 2$ is an integer. A subset $A \subseteq \Fp$ is called \emph{rectifiable of order $k$} if there exists an injection $f \cl A \to \Z$ such that, for any $a_1, \cdots a_k, a_1', \cdots a_k' \in A$, $a_1 + \cdots + a_k = a_1' + \cdots + a_k'$ if and only if $f(a_1) + \cdots + f(a_k) = f(a_1') + \cdots + f(a_k')$. When $k = 2$, we simply say that $A$ is \emph{rectifiable}. 
		
		A pair of subsets $(A,B)$ of $\Fp$ is called \emph{rectifiable}, if there exist injections $f \cl A \to \Z$ and $g \cl B \to \Z$ such that, for any $a, a' \in A$ and $b, b' \in B$, $a + b = a' + b'$ if and only if $f(a) + g(b) = f(a') + g(b')$. 
	\end{definition}
	
	We use the following results from~\cite{BLT22}. We write out the parameters $\eps, \alpha$ explicitly in \Cref{thm:BLT_stab_2pts} and modify the statement of \Cref{thm:BLT_general_3pts} in order to adapt our usage. The proof of \Cref{thm:BLT_general_3pts} can be found in \Cref{apdx}. 
	
	\begin{theorem} [{\cite[Theorem 25]{BLT22}}] \label{thm:BLT_stab_2pts}
		For all $\eps, \gamma>0$ and $t \ge 2^{10}$ there exist $\delta, \alpha>0$ such that the following holds. 
		Let $A$ and $B$ be subsets of $\Fp$. Suppose that
			\[ 2\le |B| \le \alpha |A| \text{ and } |A|+|B|\le (1-\eps)p \]
		and
			\[ \max_{b_1, b_2 \in B}|A+\{b_1, b_2\}| \le |A|+|B|-1+\delta |B|. \]
			
		Then there exist arithmetic progressions $P$ and $Q$ with the same common difference, and sizes at least $t|B|$ and at most $\lfloor(1+\gamma)|B|\rfloor$ respectively, such that $|A \cap  P| \le \gamma |B|$ and  $B \subseteq Q$. 
		
		Moreover, $\delta, \alpha$ can be chosen as $\delta = \min(2^{-13} \cdot 3.1 \cdot 10^{-1549}, 2^{-3} \gamma)$, $\alpha = \eps \cdot \min(2^{-5} \delta, 2^{-5}t^{-2} \gamma)$. 
	\end{theorem}
	
	\begin{theorem} [{\cite[Theorem 26]{BLT22}}] \label{thm:BLT_general_3pts}
		Let $A$ and $B$ be subsets of $\Fp$, and let $r \ge 0$ be an integer. Suppose that $I=[p_l, p_r]$ and $J=[q_l, q_r]$ are intervals of $\Fp$ satisfying
			\[ |A| \ge 8r ,\ 2\le |B| \le \min(|A|-2r, p/2^{11}) ,\ |I|=(2^{10}+2r)|B| \ \text{and} \ |J| \le (1+2^{-10})|B| \]
		and
			\[ |A \cap I| \le 2^{-10} |B| \ \text{and} \ \{q_l,q_r\} \subseteq B \subseteq J. \]
			
		Then either there exist $b_1, b_2, b_3 \in B$ such that $|A+\{b_1, b_2, b_3\}|\ge |A|+|B|-1+r$ or the pair $(A,B)$ is rectifiable. 
	\end{theorem}
	
	\begin{theorem} [{\cite[Theorem 27]{BLT22}}] \label{thm:BLT_strong_stab_B_Zp}
		For every $\eps > 0$, there exists $\delta > 0$ such that for every $\alpha > 0$, there is a value $c$ for which the following holds. Let $A$ and $B$ be subsets of $\Fp$. Suppose that
			\[ 2 \le \min(|A|,|B|) ,\ \alpha |B| \le |A| \le \alpha^{-1} |B| \ \text{and} \ |A|+|B| \le (1-\eps)p \]
		and
			\[ \max_{B'\in B^{(c)}} |A+B'| = |A|+|B|-1+r \le |A|+|B|-1+ \delta \min(|A|,|B|) \]
		for some integer $r$. 
		
		Then $B$ is contained in an arithmetic progression of size $|B|+r$. Here $B^{(c)}$ stands for the set of all subsets of $B$ of size at most $c$. 
	\end{theorem}
	
	Green and Ruzsa proved the following theorem, which states that if both the doubling constant and the size of a subset $A$ of $\F_p$ are small, then $A$ is rectifiable. 
	
	\begin{theorem}[\cite{GR06}] \label{thm:Green--Ruzsa}
		Suppose $p$ is a prime number, $k \ge 2$ is an integer, and $A \subseteq \F_p$ of size $\alpha p$. 
		If $|A+A| \le K|A|$ and $\alpha \le (16kK)^{-12K^2}$, then $A$ is rectifiable of order $k$. 
	\end{theorem}
	
\section{Proofs} \label{sec:main_proof}
	\begin{proof} [Proof of \Cref{thm:main_Z_case}]
		(i) Since $|A| \ge |B|$ and the degree of each element in $B$ is at most $\D$, there exists $a \in A$ such that the degree $d(a) \le \D$. 
		
		Let $A = \{a_1, \cdots, a_m\}$ and $B = \{b_1, \cdots, b_n\}$, where $a_1 < a_2 < \cdots < a_m$ and $b_1 < b_2 < \cdots < b_n$. 
		
		If $d(a_1) \le \D$, consider the following $|A| + |B| - 1$ elements: 
			\[ a_1 + b_1, a_1 + b_2,\cdots,a_1 + b_n,\ a_2 + b_n,\cdots \ ,a_m + b_n, \]
		
		At most $2\D$ of these elements are missing from $A \rplus B$, so we obtain $|A \rplus B| \ge |A| + |B| - 1 - 2\D$. 
		
		If $d(a_1) > \D$, then by the averaging argument, there exists $a_i \in A$ such that $d(a_i) < \D$. Consider the following $|A|+|B|-1$ elements: 
			\[ a_1+b_1, a_2+b_1, \cdots, a_i+b_1,\ a_i+b_2, \cdots, a_i+b_n,\ a_{i+1}+b_n, \cdots, a_m+b_n. \]
		
		At most $d(b_1) + d(a_i) + d(b_n) \le 3\D - 1$ of these elements are missing from $A \rplus B$. Thus, $|A \rplus B| \ge |A + B| - (3\D - 1) \ge |A| + |B| - 3\D$. 
		
		(ii) The argument follows similarly to the first case in (i). 
	\end{proof}
	
	\begin{remark}
		There is a theorem from~\cite{BLT22} states that if $A$ and $B$ are two finite non-empty subsets of $\Z$ with $|A| \ge |B|$, then there exist elements $b_1,b_2,b_3 \in B$ such that $|A+\{b_1,b_2,b_3\}| \ge |A|+|B|-1$. 
		
		\Cref{thm:main_Z_case}(i) could also be derived with a minor modification in the proof of this theorem. However, for our specific purpose, the proof presented here is more direct and less involved. 
	\end{remark}
	
	\begin{remark}
		The bound $|A| + |B| - 1 - 2\D$ in \Cref{thm:main_Z_case}(ii) is tight. Consider the construction where $A = B = [n]$ and define $\R = \left( [\D] \times [\D] \right) \cup \left( \left\{ n-\D+1, \cdots, n \right\} \times \left\{ n-\D+1, \cdots, n \right\} \right)$. 
		
		In this case, $A+B = \{ \D+2, \D+3, \cdots, 2n-\D \}$, giving $|A + B| = 2n - 1 - 2\D$, which matches our bound exactly. 
	\end{remark}
	
	We need the following lemma showing any sufficiently small subset of $\Fp$ is rectifiable. 
	
	\begin{lemma} \label{lem:small_subset_rect}
		For any integer $n > 0$, if $p$ is a prime number satisfying $p > 4^n$, then any subset of $\Fp$ of size at most $n$ is rectifiable. 
	\end{lemma}
	
	\begin{proof}
		Let $X \subseteq \Fp$ be a subset of size at most $n$, We aim to show that there exists some $t \in \Fp^\times$ such that all elements in $tX$ have residues within the interval $[-\frac{p}{4}, \frac{p}{4}]$. This ensures that if $a+b \equiv c+d \pmod{p}$, then the equality holds in $\Z$ as well, proving rectifiability. 
		
		Consider the set of $p$ of vectors in $\{0,1,\cdots, p-1\}^{|X|}$, where the $i$-th vector is given by $(ix_1 \!\!\!\mod{p},\, ix_2 \!\!\!\mod{p},\, \cdots,\, ix_n \!\!\!\mod{p})$. Since $p > 4^n$, the pigeonhole principle guarantees that there exist distinct indices $i, j$ such that 
			\[ (ix_1 \hspace{-.8em}\mod{p},\, ix_2 \hspace{-.8em}\mod{p},\, \cdots,\, ix_n \hspace{-.8em}\mod{p}) - (jx_1 \hspace{-.8em}\mod{p},\, jx_2 \hspace{-.8em}\mod{p},\, \cdots,\, jx_n \hspace{-.8em}\mod{p}) \in \left[-\frac{p}{4}, \frac{p}{4}\right]^{|X|}. \]
		
		Setting $t = i - j$, we conclude that the function $f \cl x \mapsto \overline{tx}$ is a rectifying map for $A$, where $\overline{x}$ denotes the least absolute residue of $x$ modulo $p$. 
	\end{proof}
	
	Let $\alpha$ be the constant given in \Cref{thm:BLT_stab_2pts}, whose explicit value will be determined later. The proof of \Cref{thm:main_Fp_case} is divided into two cases: If $|B| \ge \alpha |A|$, we establish \Cref{thm:main_Fp_case} using a rectifiability argument due to Green and Ruzsa~\cite{GR06}. If $|B| < \alpha |A|$, we apply a modified technique developed by Bollob\'as, Leader, and Tiba~\cite{BLT22}.
	
	\begin{proof}[Proof of \Cref{thm:main_Fp_case}]
		Let $\alpha = 2^{-18} \cdot 3.1 \cdot 10^{-1549} \eps$ and take $c_\eps = (1 + \frac{1}{\alpha})^{-1} \left( 8(2\D + 3 + \frac{1}{\alpha}) \right)^{-24(2\D + 3 + \frac{1}{\alpha})^4}$, $p_0 = 4^{10\D}$. We will show that \Cref{thm:main_Fp_case} holds for $c_\eps$ and $p_0$. 
		
		We prove (i) and (ii) simultaneously by contradiction. Suppose that the conclusion does not hold, i.e. $|A \rplus B| < |A| + |B| -1 - 2\D$. 
		
		\textbf{Case 1:} $|B| \ge \alpha |A|$. 
		
		We aim to derive a contradiction by finding a rectifying map that lifts $A \cup B$ to $\Z$, allowing us to apply \Cref{thm:main_Z_case}. 
		
		Since each element in $B$ is connected to at most $\D$ elements in $A$ with respect to $\R$, we have
			\[ |A+B| \le |A \rplus B| + \D |B| < |A| + |B| - 1 - 2\D + \D |B| < \min \left\{ (\D+2)|A| , (\D + 1 + \frac{1}{\alpha})|B| \right\}. \]
		
		By Pl\"unneke--Ruzsa's inequality, it follows that 
			\[ |A+A| \le (\D+1 + \frac{1}{\alpha})^2 |B|, \quad |B+B| \le (\D+2)^2 |A|. \]
		
		Thus, 
		\begin{align*}
			| (A \cup B) + (A \cup B) | \le |A+A| + |B+B| + |A+B| &\le (\D+1 + \frac{1}{\alpha})^2 |B| + (\D+2)^2 |A| + (\D+2)|A| \\
				&\le (2\D + 3 + \frac{1}{\alpha})^2 |A| \\
				&\le (2\D + 3 + \frac{1}{\alpha})^2 |A \cup B|. 
		\end{align*}
		
		Since $\alpha |A| \le |B|$, we have 
			\[ |A \cup B| \le (1 + \frac{1}{\alpha}) |B| \le (1 + \frac{1}{\alpha}) c_\eps p = \left( 8(2\D + 3 + \frac{1}{\alpha}) \right)^{-24(2\D + 3 + \frac{1}{\alpha})^4}p. \]
			
		By \Cref{thm:Green--Ruzsa}, $A \cup B$ is rectifiable. 
		
		Let $f \cl A \cup B \to \Z$ be a rectifying map. We lift the relation $\R \subseteq A \times B$ to $\widetilde{\R} \subseteq f(A) \times f(B)$ along $f$. Then we have $|A \rplus B| = |f(A) \rtplus f(B)|$. Applying \Cref{thm:main_Z_case} we get the desired result, thus completing this case.
		
		\textbf{Case 2:} $|B| < \alpha |A|$. 
		
		If $|B| = 1$, the result is trivial. Assume $|B| \ge 2$. If there exist $b_1, b_2 \in B$ such that $|A + \{b_1, b_2\}| \ge |A| + |B| - 1$, then, by degree boundedness, we have $|A \rplus B| \ge |A| + |B| -1 - 2\D$. Setting\footnote{It is worth noting that we need to verify that the $\alpha$ defined here meets the requirements of \Cref{thm:BLT_stab_2pts}. } $\gamma = 2^{-10}$ and $t = 2^{10} + 2\D$ in \Cref{thm:BLT_stab_2pts}, we obtain arithmetic progressions $P$ and $Q$ with the same common difference satisfying 
			\[|P| \ge (2^{10} + 2\D)|B|, \quad |Q| \le \lfloor (1+2^{-10})|B| \rfloor, \quad |A \cap P| \le 2^{-10} |B|,\quad \text{and} \quad B \subseteq Q.\]
		
		By rescaling $A, B$ and shrinking $Q$ if necessary, we may assume the common difference of $P$ and $Q$ is $1$, and the endpoints of $Q$  belonging to $B$. 
		
		If $|A| < 8\D$, then $|A \cup B| < (1 + \alpha)|A| < 10\D$. Furthermore, since the second assumption $p > p_0$ holds in this case, by \Cref{lem:small_subset_rect}, we know that $A \cup B$ is rectifiable. This allows us to apply \Cref{thm:main_Z_case} and conclude the proof. 
		
		If $|A| \ge 8\D$, by \Cref{thm:BLT_general_3pts}, we conclude that either there exist $b_1, b_2, b_3 \in B$ such that $|A + \{b_1, b_2, b_3\}| \ge |A| + |B| - 1 + \D$, or $(A, B)$ is a rectifiable pair. In the first case, we have $|A \rplus B| \ge |A \rplus \{b_1, b_2, b_3\}| \ge |A| + |B| - 1 - 2\D$. In the latter case, applying again \Cref{thm:main_Z_case} yields the desire result, completing the proof. 
	\end{proof}
	
	\begin{proof} [Proof of \Cref{thm:main_Fp_case_strong}]
		Assume without loss of generality that $|B| \le |A|$. Define $\alpha = 2^{-18} \cdot 3.1 \cdot 10^{-1549} \eps$. The case where $|B| < \alpha |A|$ has already been addressed in the proof of \Cref{thm:main_Fp_case}, so we may assume that $|B| \ge \alpha |A|$. 
		
		Let $c$ and $\delta$ be the constants given by \Cref{thm:BLT_strong_stab_B_Zp} for parameters $\eps$ and $\alpha$. Without loss of generality, assume $\delta < 1$. Define $r = c\D$ and set $\beta = \frac{c}{\delta \alpha}$, $p_0 = 4^{2\beta\D}$. Now, assume either $|A| \ge \beta\D$ or $p > p_0$. 
		
		If $|A| < \beta\D$, we have $p > p_0 = 4^{2\beta\D}$ and $|A \cup B| \le |A| + |B| \le 2|A| < 2\beta\D$. Hence, by \Cref{lem:small_subset_rect}, $A \cup B$ is rectifiable. Then apply \Cref{thm:main_Z_case}, we get the desired result. 
		
		Now assume that $|A| \ge \beta\D$. In this case, we have $r \le \delta \cdot \min (|A|, |B|) = \delta |B|$, since otherwise we would have 
			\[ |A| \le \frac{1}{\alpha} |B| < \frac{r}{\delta \alpha} = \frac{c\D}{\delta \alpha} = \beta\D, \]
		a contradiction. 
		
		Besides, since $p \ge |A|$, we have
		\begin{equation} \label{eqn:eps_p}
			\eps p \ge \eps|A| \ge \eps\beta\D = \frac{\eps c\D}{\delta\alpha} > (3c+1)\D = 3r + \D.
		\end{equation}
		
		Let $A^{(c)}$ denote the set of all subsets of $A$ of size at most $c$. We may further assume that 
			\[ \max_{A'\in A^{(c)}}|A'+B| \le |A|+|B|-1+r, \] 
		since otherwise, we would have 
		\begin{align*}
			\max_{A' \in A^{(c)}}|A' \rplus B| &\ge \max_{A' \in A^{(c)}} |A' + B| - c\D \ge |A| + |B| - 1 + r - c\D \\
				&= |A| + |B| - 1 \\
				&> |A| + |B| - 1 - 2\D,
		\end{align*}
		which would directly lead to the desired result. 
		
		Thus, the assumptions\footnote{Here, we have swapped the labels of sets $A$ and $B$ when applying \Cref{thm:BLT_strong_stab_B_Zp}. } of \Cref{thm:BLT_strong_stab_B_Zp} hold, implying that $A$ is contained in an arithmetic progression $I$ of size $|A| + r$. By multiplying an element from $\Fp^\times$, we may assume $I$ is an interval. 
		
		If there exist two elements $b_1, b_2 \in B$ such that the distance between them (i.e., $\min(b_1 - b_2 \hspace{-.4em}\mod{p},\, b_2 - b_1 \hspace{-.4em}\mod{p})$) is at least $|B| + r$, then we have
			\[ |A + \{b_1, b_2\}| \ge \min (|I| + |B| + r - 2r, 2|A|) \ge |A| + |B|. \]
		
		Thus, 
			\[ |A \rplus B| \ge |A \rplus \{b_1, b_2\}| \ge |A| + |B| - 2\D. \]
		
		If there is no such pair, suppose $b_1, b_2 \in B$ be the elements with the maximum distance among elements of $B$, and let $J$ be the minor arc between them (i.e., the interval ending at $b_1$ and $b_2$ with size smaller than $\frac{p}{2}$). 
		
		If $B \subseteq J$, then $B$ is contained in an interval of length at most 
			\[ |J| \le |B| + r - 1 < p - |A| - r = p - |I|, \]
		where the second inequality follows from $|A| + |B| \le (1-\eps)p$ and (\ref{eqn:eps_p}). By translating $I$ and $J$, we may assume that both intervals start from $0$. In this case, the sum of the maximum elements of $A$ and $B$ is smaller than $p$. Therefore, $(A, B)$ forms a rectifiable pair. 
		
		If $B \nsubseteq J$, let $b_3 \in B \setminus J$. By the maximality of the distance between $b_1$ and $b_2$, the minor arcs between $b_1 b_3$ and $b_2 b_3$ are not contained in $J$. Hence, the union of the minor arcs between $b_1b_2$, $b_2b_3$, and $b_1b_3$ covers all of $\Fp$, and the length of each minor arc is at most $|B| + r \le |A| + r$. This implies that $I + \{b_1, b_2, b_3\} = \Fp$. Therefore, 
		\begin{align*}
			|A \rplus B| &\ge |A \rplus \{b_1, b_2, b_3\}| \ge |I + \{b_1, b_2, b_3\}| - 3r - 3\D \\ 
			&\ge |\Fp| - 3r - 3\D = p-3r - 3\D \\
			&\ge |A| + |B| - 1 - 2\D, 
		\end{align*} 
		where the last inequality again follows from $|A| + |B| \le (1-\eps)p$ and (\ref{eqn:eps_p}). 
	\end{proof}
	
\section{Constructions} \label{sec:constructions}
	We prove \Cref{thm:Z_construction} first. Recall that in proving \Cref{thm:Z_construction}, we need to construct two subsets $A, B \subseteq \Z$, both of size $n$, along with a relation $\R$ of bounded degree $\D$ on $B$, such that $|A \rplus B| = |A| + |B| - 1 - \left\lfloor \frac{5\D}{2} \right\rfloor$. 
	
	\begin{proof} [Proof of \Cref{thm:Z_construction}]
		Let $r = \lfloor \frac{\D}{2} \rfloor$. Define 
			\[ A = \{1,2,\cdots, n-\D, n-\D+r+1, \cdots, n+r\},\quad B = \{1,2,\cdots, n\}. \]
		
		Then $|A| = |B| = n$, and $A + B = \{2, 3, \cdots, 2n+r\}$. Define 
			\[ C \eqdef \{n-\D+1, \cdots, n-\D+r\} + \{1,n\} = \{n-\D+2, \cdots, n-\D+r+1\} \cup \{2n-\D+1, \cdots, 2n-\D+r\}, \]
			\[ D \eqdef \{2,\cdots,\D+1\} \cup \{2n+r-\D+1, \cdots 2n+r\}. \]
			
		Then $C, D$ are disjoint subsets of $A+B$, and for any $b \in B$ we have $|(A + b) \cap C| + |(A + b) \cap D| \le \D$. We associate each $b$ with those $a \in A$ such that $a + b \in C \cup D$ in $\R$. Then each element in $B$ connects with at most $\D$ elements in $A$, and 
			\[ |A \rplus B| = |(A+B) \setminus (C \cup D)| = (2n+r-1) - (2r + 2\D) = 2n-1-\left\lfloor \frac{5\D}{2} \right\rfloor. \]
	\end{proof}
	
	The idea behind this construction is to remove $r$ additional elements from the trivial construction, which achieves $|A \rplus B| = |A| + |B| - 1 - 2\D$ by taking $A$ and $B$ as two intervals and deleting $\D$ elements near both endpoints of $A+B$. 
	
	Let $r$ be an arbitrary integer for now. The bottleneck of this construction is that there is a translation $A + b$ containing $C$, which consists of the union of two intervals $C_1 \cup C_2$ with total size $2r$. This forces us to impose the condition $r \le \lfloor \frac{\D}{2} \rfloor$. 
	
	One may attempt to improve this by removing $r$ elements evenly across $A$ rather than a contiguous interval. If we do this, each translation of $A$ would cover at most $r+1$ deleted elements. This seems to be a wiser approach since we can take $r = \D - 1$, yielding $|A| = |B| = n$ and $|A \rplus B| = 2n-3\D$. However, at this point, we would fail to remove $\D$ elements around the the endpoints of $A + B$ from $A \rplus B$. This explains why we have to remove an interval and why $|A \rplus B| = |A| + |B| - 1 - \left\lfloor \frac{5\D}{2} \right\rfloor$ is optimal concerning this idea. 
	
	We believe this construction is tight. A potential proof approach is to apply stability results for $|A + B| \ge |A| + |B| - 1$ to first establish that $A$ and $B$ look like intervals. Next, we analyze $|A + \{b_{\min}, b_{\max}\}|$ which should approximately as large as $|A| + |B|$. Then we study the translations $A + b$ with $b$ is near $b_{\min}, b_{\max}$ and where $b$ is an appropriately chosen interior element from $B$ respectively. 
	
	We now prove \Cref{thm:counterexample_Lev}(i). 
	
	\begin{proof}[Proof of  \Cref{thm:counterexample_Lev}(i)]
		Let 
			\[ A = \{0, k, 2k, \cdots, \ell k, (\ell+1)k, (\ell+1)k+1, \cdots, p-1\},\ B = \{0, -1, -2, \cdots, -(\ell k + 1)\}. \] 
		
		Then we have $|A| = p - (\ell+1)k + \ell + 1 = p - (k-1)\ell - k + 1$ and $|B| = \ell k + 2$. 
		
		Denote $C = \{0,1, \cdots, k-1\}$. For each $b \in B$, we have $|(A + b) \cap C| = 1$. Hence, there is a way to associate each element in $B$ with an element $a \in A$ when constructing $\R$ so that $(A \rplus B) \cap C = \varnothing$. This implies $|A \rplus B| \le |\Fp \setminus C| = p-k$. 
	\end{proof}
	
	Although this construction is quite simple, I would like to briefly explain the underlying idea. The key insight is to first consider the \emph{avoiding set} $\mF \eqdef \Fp \setminus (A \rplus B)$. Given $A$ and $\mF$, we know that $B$ must be a subset of $\{b \cl |(\mF - b) \cap A| \le 1\}$. For optimality, we may take $B$ to be the entire set. This reduces the problem to the following: Given the sizes of $A$ and $\mF$, determine the maximum size of $\{b \cl |(\mF - b) \cap A| \le 1\}$. 
	
	Suppose $k = |\mF|$, and let $r_i$ denote the number of translations of $\mF$ such that $|(\mF + x) \cap A| = i$. Then, we have
	\begin{equation} \label{eqn:sum_of_r}
		r_0 + r_1 + \cdots + r_k  = p, 
	\end{equation}
	and
	\begin{eqnarray} \label{eqn:weighted_sum_of_r}
		r_1 + 2r_2 + \cdots + kr_k = k|A|. 
	\end{eqnarray}
	
	Moreover, we know that $|B| = r_0 + r_1$. 
	
	Let us examine the case where $k = 3$. Consider an arbitrary subset $A \subseteq \Fp$ of size approximately $\frac{p}{2}$, and let $\mF$ be a subset of $\Fp$ of size $3$. Define the sequences $(r_0, r_1, r_2, r_3)$ and $(r_0', r_1', r_2', r_3')$ corresponding to $A$ and $\Fp \setminus A$ respectively. 
	
	Since a translation of $\mF$ that contains exactly $i$ elements of $A$ must contain exactly $3-i$ elements of $\Fp \setminus A$, we obtain $(r_0', r_1', r_2', r_3') = (r_3, r_2, r_1, r_0)$. Furthermore, adding a new element to $A$ decreases both $r_0$ and $r_1$ by at most $3$. 

	Consequently, if we always have $|A \rplus B| \ge p-2$ for any subsets $A, B$ satisfying $|A| + |B| \ge p + O(1)$ and any relation $\R$ a mapping from $B$ to $A$, then we must have 
		\[ r_0 + r_1 \le \frac{p}{2} + O(1),\quad r_2 + r_3 \le \frac{p}{2} + O(1),\]
	which implies 
		\[ \frac{p}{2} - O(1) \le r_0 + r_1 \le \frac{p}{2} + O(1) \tag{4} \label{eqn:r0+r1}\]
	for any $A \subseteq \Fp$ of size $\frac{p}{2} \pm O(1)$. 
	
	Fixing $\mF$, this condition is highly restrictive, as each $r_i$ is the sum of indicator functions for the event that a translation of $\mF$ contains exactly $i$ elements of $A$. Each of these indicator functions depends only on a constant number of elements being included or excluded from $A$, meaning that selecting an element $a \in A$ affects only a constant number of terms in the summation. Consequently, the distribution of $r_0 + r_1$ should approximately follow a normal distribution, suggesting that the condition (\ref{eqn:r0+r1}) is too rigid to hold universally. 
	
	The following lemma is helpful in understanding the constructions of \Cref{thm:counterexample_Lev}(ii) and (iii). 
	
	\begin{lemma} \label{lem:large_density_cover}
		There exists a set $A \subseteq \Z$ with period $11$ and density $\frac{6}{11}$ such that for each element $a \in A$, there exists a translation of $\{0,1,4\}$ that intersects $A$ only at $a$. 
	\end{lemma}
	
	\begin{proof}
		Consider the set $A = \{0,1,2,3,5,6\} + 11\Z$. Then, we observe that 
			\[ (\{0,1,4\} + (-4)) \cap A = \{0\},\ (\{0,1,4\} + (-3)) \cap A = \{1\},\ (\{0,1,4\} + (-2)) \cap A = \{2\}, \]
			\[ (\{0,1,4\} + 3) \cap A = \{3\},\ (\{0,1,4\} + 4) \cap A = \{5\},\ (\{0,1,4\} + 6) \cap A = \{6\}. \]
			
		Thus, $A$ satisfies the required condition. 
	\end{proof}
	
	For the case where the degree is bounded on both sides, adding the constraint that each element in $A$ has degree at most $1$ further reduces the maximum possible size of $B$ to $r_0 + |\{a \in A \cl \exists x \text{ such that } (\mF + x) \cap A = \{a\}\}|$. The second term is at most $r_1$. This loss is caused by each element in $A$ can be only connected to at most $1$ element in $B$. 
	
	Lev~\cite{Lev00-ii} provided a construction with $|A| = \frac{p+1}{2}$ and $|A \rplus A| = p-3$, demonstrating that the condition $|A| + |B| \ge p+1$ is insufficient to guarantee $|A \rplus B| \ge p-2$. The sum of sizes of this construction has one extra gain from $p$. The key idea behind \Cref{thm:counterexample_Lev}(ii) and (iii) is that we can replicate Lev's construction to amplify this gap, given that the initial sum $|A| + |B|$ exceeds $p$. However, if the initial value of $|A| + |B| - p$ were negative, this approach would offer no advantage. This is why such techniques are effective only when $|A| + |B| > p$, and cannot be extended to the $|A| + |B| = (1-\eps)p$ regime in the first construction. 
	
	We are now ready to prove \Cref{thm:counterexample_Lev}(ii) and (iii) 
	
	\begin{proof}
		We begin with the proof of \Cref{thm:counterexample_Lev}(ii). 
		
		The construction differs slightly, depending on whether $\lfloor \frac{p}{11} \rfloor$ is even or odd. 
		
		If $\lfloor \frac{p}{11} \rfloor$ is even, suppose $\lfloor \frac{p}{11} \rfloor = 2t$. Set 
			\[ A = \Big( \{0,1,2,3,5,6\} + 11 \cdot \{-(t-1), -(t-2), \cdots, -1,0,1, \cdots, t-1\} \Big) \cup \big\{-11t+3, -11t+5, -11t+6 \big\}, \] 
		and define $\R$ as
		\begin{align*}
			\R &= \bigcup_{i = -t+1}^{t-1} \big\{(11i, -11i+2), (11i+1, -11i+1), (11i+2, -11i) \big\} \\
			&\qquad\quad \cup \bigcup_{i = -t}^{t-1} \big\{(11i+3, -11(i+1)+6), (11i+5, -11(i+1)+5), (11i+6, -11(i+1)+3) \big\}. 
		\end{align*} 
		
		If $\lfloor \frac{p}{11} \rfloor$ is odd, suppose $\lfloor \frac{p}{11} \rfloor = 2t+1$. Set
			\[ A = \Big( \{0,1,2,3,5,6\} + 11 \cdot \{-t, -(t-2), \cdots, -1,0,1, \cdots, t-1\} \Big) \cup \big\{11t, 11t+1, 11t+2 \big\}, \] 
		and define $\R$ as
		\begin{align*}
			\R &= \bigcup_{i = -t}^{t} \big\{(11i, -11i+2), (11i+1, -11i+1), (11i+2, -11i) \big\} \\
			&\qquad\quad \cup \bigcup_{i = -t}^{t-1} \big\{(11i+3, -11(i+1)+6), (11i+5, -11(i+1)+5), (11i+6, -11(i+1)+3) \big\}. 
		\end{align*}
		
		In both cases, $|A| = 6 \lfloor \frac{p}{11} \rfloor - 3$, $\R$ is a symmetric matching within $A$, and $|A \rplus A| = |\Fp \setminus \{-2,-1,2\}| = p-3$. 
		
		We now proceed to prove (iii). Without loss of generality, we may assume $\eps \le \frac{6}{11}$. Let $t = \lfloor \frac{\eps p}{12} \rfloor$ and $\delta = \frac{\eps}{6}$. Set
			\[ B = \Big( \{0,1,2,3,5,6\} + 11 \cdot \{-(t-1), -(t-2), \cdots, -1,0,1, \cdots, t-1\} \Big) \cup \big\{-11t+3, -11t+5, -11t+6 \big\}, \]
			\[ A = B \cup \{11t, 11t+1, \cdots, p - 11t-1\}, \]
		and define $\R$ as
		\begin{align*}
			\R &= \bigcup_{i = -t+1}^{t-1} \big\{(11i, -11i+2), (11i+1, -11i+1), (11i+2, -11i) \big\} \\
			&\qquad\quad \cup \bigcup_{i = -t}^{t-1} \big\{(11i+3, -11(i+1)+6), (11i+5, -11(i+1)+5), (11i+6, -11(i+1)+3) \big\}. 
		\end{align*} 
		
		Then we have $|A| = (1-\frac{5\eps}{6})p - O(1)$, $|B| = \eps p - O(1)$, and $A \rplus B = \Fp \setminus \{-2,-1,2\}$. This completes the proof. 
	\end{proof}
	
	\begin{remark}
		If we could replace the set $\{0,1,4\}$ in \Cref{lem:large_density_cover} with any set of integers of size $4$ or more while keeping the upper density of $A$ greater than $\frac{1}{2}$, then the constructions in \Cref{thm:counterexample_Lev}(ii) and (iii) could be improved to satisfy $|A \rplus B| \le p-4$. 
	\end{remark}
	
	\begin{proof} [Proof of \Cref{thm:sum_greater_than_p_case}]
		We prove the contrapositive. Suppose $|A \rplus B| \le p-k$. Let $\mF$ be a subset of $\Fp \setminus (A \rplus B)$ of size $k$. Denote $r_i \eqdef \left| \left\{ x \in \Fp \cl |(\mF + x) \cap A| = i \right\} \right|$. Then we have (\ref{eqn:sum_of_r}) and (\ref{eqn:weighted_sum_of_r}). Hence
			\[ r_1 + 2r_2 + \cdots + kr_k \le (r_0 + r_1) + k(r_2 + \cdots + r_k) = (r_0 + r_1) + k(p - r_0 - r_1) = kp - (k-1)(r_0 + r_1). \]
		
		So we have
			\[ kp - (k-1)(r_0 + r_1) \ge k|A|. \]
		
		Since $|B| \le r_0 + r_1$, we conclude
		\begin{align*}
			|A| + |B| &= |A| + \frac{2k-2}{2k-1}|B| + \frac{1}{2k-1}|B| \le |A| + \frac{2k-2}{2k-1}(r_0 + r_1) + \frac{1}{2k-1}|A| \\
			&\le |A| + \frac{2k-2}{2k-1} \cdot \frac{kp - k|A|}{k-1} + \frac{1}{2k-1}|A| \\ 
			&= \frac{2kp}{2k-1}. 
		\end{align*}
		
		This completes the proof. 
	\end{proof}
	
\section{Concluding Remarks}
	In this paper, we studied general restricted sumsets over $\Fp$ and compared their behavior with that in the Erdős–Heilbronn theorem. When $\R$ is a matching, the two settings behave similarly in the regime $|A| + |B| \le (1-\eps)p$; however, when $|A| + |B| > p$, we provided an example illustrating their difference. Nevertheless, the avoiding set $\mF \ce \Fp \setminus (A +_\R B)$ remains an interesting object of study in the range $|A| + |B| > p$. It was shown in~\cite{Lev00-ii} that $\mF$ is a Sidon set and hence satisfies the bound $|\mF| < \sqrt{p} + \frac{1}{2}$, which is currently the best known. Any construction achieving $|\mF| = \omega(1)$, or any improvement proving $|\mF| = o(\sqrt{p})$, would be of interest. 
	
\section*{Acknowledgements}
	I would like to thank Vsevolod Lev for several helpful comments on earlier versions of this paper and for bringing \cite{BLT22} to my attention. I also benefited from valuable discussions with Boris Bukh and Zichao Dong. This work was initiated during the $2^{\text{nd}}$ IBS ECOPRO Student Research Program in the summer of 2024. I am grateful to Hong Liu for hosting my visit to IBS as a student researcher. 
	
	\bibliographystyle{alpha}
	{\small \bibliography{reference}}
	
\appendix
\section{Proof of \texorpdfstring{\Cref{thm:BLT_general_3pts}}{Theorem~\ref{thm:BLT_general_3pts}}} \label{apdx}
	\begin{proof}
		By the assumptions, there is a partition $I = I_{-2}\sqcup I_{-1}\sqcup I_0\sqcup I_1\sqcup I_2$ of the interval $I$ into five consecutive disjoint subintervals, such that $p_l \in I_{-2} \text{ and } p_r \in I_2$, $ |I_0|=|I_{-2}|= |I_{2}| = |J|$, and $ |I_0\cap A| \le 2^{-4} |(I_{-1} \sqcup I_0 \sqcup I_1)\cap A|$. 
		This partition can be constructed by placing $I_{-2}, I_2$ at the endpoints of $I$ and selecting $I_0$ as a random subinterval within $I \setminus (I_{-2} \sqcup I_2)$. 

		Let $I_{\infty}$ be the complement of $I$, i.e. $I_{\infty}=I^c$. For a subset $S\subseteq \{-2,-1,0,1,2, \infty\}$ write $A_S = \sqcup_{x\in S} (A \cap I_x)$. 
		By the construction, for any $x, y \in B$, we have the following three assertions:
			\[ \text{the sets} \ \ A_{-1}+q_r,\ A_1+q_l \ \ \text{and} \ \ A_{-2,\infty, 2}+y \ \ \text{are disjoint}, \]
			\[ \text{the sets} \ \ A_{-1,0,1}+x \ \ \text{and} \ \  A_{\infty}+y \ \ \text{are disjoint}, \]
		and
			\[ |A_0| \le 2^{-4}|A_{-1,0,1}|. \]
		
		From the proof of \cite[Theorem 25]{BLT22} we have the following two claims. Here, $\E_{b \in \ast}|\mathrel{\raisebox{0.3ex}{\scalebox{0.5}{$\bullet$}}}|$ denotes the expectation of the size of the set $\mathrel{\raisebox{0.3ex}{\scalebox{0.5}{$\bullet$}}}$ when $b$ is chosen uniformly at random from $\ast$. We decompose the value $|A + X|$ for a given set $X$ using the equation $|A + X| = |A_{-2,\infty,2} + X| + |(A_{-1,0,1} + X) \setminus (A_{-2,\infty,2} + X)|$, and estimate each term separately. 
		
		\textbf{Claim A.} \emph{
			If $|A_{-2,\infty,2}| \ge |B|-1$ then
				\[ \E_{b \in B \setminus \{q_r\}} \bigg| A_{-2,\infty,2} + \{q_l,q_r, b\} \bigg| \ge |A_{-2,\infty,2}| + |B| - 1. \]
			If $|A_{-2,\infty,2}| \le |B|-1$ then either
				\[ \E_{b \in B \setminus \{q_r\}} \bigg| A_{-2,\infty,2} + \{q_l,q_r, b\} \bigg| \ge 2|A_{-2,\infty,2}| + 2^{-3} (|B| - 1 - |A_{-2,\infty,2}|) \]
			or
				\[ \E_{b_2,b_3\in B \setminus \{q_r\}} \bigg| A_{-2,\infty,2} + \{q_l,b_2, b_3\} \bigg| \ge 2.5 |A_{-2,\infty,2}|. \]
		}
		
		\textbf{Claim B.} \emph{
			We have
				\[ \E_{b\in B \setminus \{q_r\}}\bigg|(A_{-1,0,1}+\{q_l,q_r, b\}) \setminus ( A_{-2,\infty,2}+\{q_l,q_r, b\} ) \bigg| \ge (2-2^{-3})|A_{-1,0,1}| \]
			and
				\[ \E_{b_2,b_3\in B \setminus \{q_r\}}\bigg|(A_{-1,0,1}+\{q_l,b_2, b_3\}) \setminus ( A_{-2,\infty,2}+\{q_l,b_2, b_3\} ) \bigg| \ge (2-2^{-3})|A_{-1,0,1}|. \]
		}
		
		Now we proceed to prove \Cref{thm:BLT_general_3pts}. 
		
		\textbf{Case 1:} 
		If $|A_{-2, \infty, 2}|\ge |B|-1$. Applying Claim A and B, we have
		\begin{eqnarray*}
			&\E_{b \in B \setminus \{q_r\}} \bigg| A+\{q_l, q_r,b\} \bigg| \\
			& \ge \E_{b\in B \setminus \{q_r\}}\bigg|A_{-2,\infty,2}+\{q_l,q_r, b\} \bigg|+ \E_{b\in B \setminus \{q_r\}}\bigg|(A_{-1,0,1}+\{q_l,q_r, b\}) \setminus ( A_{-2,\infty,2}+\{q_l,q_r, b\} ) \bigg| \\
			& \ge |A_{-2, \infty, 2}|+|B|-1+(2-2^{-3})|A_{-1,0,1}| \\
			& = |A| + |B| - 1 + \frac{7}{8} \cdot |A_{-1,0,1}|. 
		\end{eqnarray*}
		
		If $|A_{-1,0,1}| \ge \frac{8r}{7}$, this expression is at least $|A| + |B| - 1 + r$. We get the desired result. 
		Otherwise, since $I_{-1,0,1}$ is an interval of size 
		\begin{align*}
			|I_{-1,0,1}| &= |I| - |I_{-2}| - |I_{2}| = |I| - 2|J| \ge (2^{10} + 2r) |B| - 2|J| \\ 
				&\ge ((1+2^{-10})^{-1}(2^{10} + 2r) - 2)|J| \ge (\frac{8r}{7}+1)|J|, 
		\end{align*}
		there exists a subinterval $I' \subseteq I$ of size at least $|J|-1$ such that $I' \cap A = \varnothing$. Translating $A$ and $B$, we may assume $J = \{0,1, \cdots, |J|-1\}$ and $\{p-|J| + 1, \cdots, p-1\} \subseteq I'$. Then the sum of $A$ and $B$ has no carry-over and hence $(A,B)$ is a rectifiable pair. 
		
		\textbf{Case 2:} 
		If $|A_{-2, \infty, 2}|< |B|$, using Claims A and B we have either
		\begin{eqnarray*}
			& \E_{b \in B \setminus \{q_r\}} \bigg| A+\{q_l, q_r,b\} \bigg| \\
			& \ge \E_{b\in B \setminus \{q_r\}} \bigg| A_{-2,\infty,2}+\{q_l,q_r, b\} \bigg| + \E_{b\in B \setminus \{q_r\}} \bigg| (A_{-1,0,1}+\{q_l,q_r, b\}) \setminus ( A_{-2,\infty,2}+\{q_l,q_r, b\} ) \bigg| \\
			& \ge 2|A_{-2,\infty,2}| + 2^{-3} (|B|-1-|A_{-2,\infty,2}|) + (2-2^{-3}) |A_{-1,0,1}| \\
			& = (2-2^{-3})|A|+2^{-3}(|B|-1) \\
			& \ge |A| + (1-2^{-3})(|B| + 2r) + 2^{-3}(|B|-1) \\
			& \ge |A|+|B|-1+r,
		\end{eqnarray*}
		or
		\begin{eqnarray*}
			& \E_{b_2, b_3 \in B \setminus \{q_r\}} \bigg| A+\{q_l, b_2,b_3\} \bigg| \\
			& \ge \E_{b_2,b_3\in B \setminus \{q_r\}} \bigg| A_{-2, \infty, 2} + \{q_l, b_2, b_3\} \bigg| + \E_{b_2, b_3 \in B \setminus \{q_r\}} \bigg| (A_{-1,0,1} + \{q_l,b_2, b_3\}) \setminus ( A_{-2, \infty, 2} + \{q_l,b_2, b_3\} ) \bigg| \\
			& \ge 2.5 |A_{-2, \infty, 2}| + (2-2^{-3}) |A_{-1,0,1}| \\
			& \ge (2.25 - 2^{-4}) |A| \\
			& \ge |A|+|B|-1+r. 
		\end{eqnarray*}
		Here we use the assumptions $|A| \ge |B| + 2r$, $|A| \ge 8r$, and the fact $|A_{-1,0,1}|  \le |A\setminus I^c| \le  2^{-1}|A|$. 
		
		This completes the proof of \Cref{thm:BLT_general_3pts}.
	\end{proof}
	
\end{document}